\documentclass[12pt]{article}
\usepackage{amsfonts, amsmath, amssymb, latexsym, eucal, amscd} 
\usepackage{cite}
\usepackage[dvips]{pict2e}
\usepackage[all]{xy}
\newtheorem{theorem}{Theorem}[section]
\newtheorem{lemma}[theorem]{Lemma}
\newtheorem{prop}[theorem]{Proposition}

\newtheorem{exAux}[theorem]{Example}

\newtheorem{Def}[theorem]{Definition}
\newenvironment{defi}{\begin{Def} \rm}{\end{Def}}
\newtheorem{Note}[theorem]{Note}
\newenvironment{note}{\begin{Note} \rm}{\end{Note}} 
\newtheorem{Problem}[theorem]{Problem}

\newtheorem{Rem}[theorem]{Remark}

\newtheorem{Not}[theorem]{Notation}

\newtheorem{Conj}[theorem]{Conjecture}

\newtheorem{Ass}[theorem]{Assumption}

\newenvironment{proof}{\medskip\noindent{\bf Proof.\ }}{\qed\medskip}
\newenvironment{proofof}[1]{\medskip\noindent{\bf Proof  of {#1}.\ 
}}{\qed\medskip}
\newcommand{\qed}{\hfill\mbox{$\Box$\qquad\qquad}}

\newcommand{\F}{{\mathbb F}}

\newcommand{\Mat}{\text{\rm Mat}}
\newcommand{\vphi}{\varphi}
\renewcommand{\th}{\theta}

\newcommand{\vth}{\vartheta}

\newif\ifDRAFT

\begin{document}

\title{Variations on a circular  Hessenberg pair}

\author{Kazumasa Nomura and Paul Terwilliger}

\maketitle

\begin{abstract}
A square matrix is called Hessenberg whenever each entry below the subdiagonal is 
zero and each entry on the subdiagonal is nonzero.
A Hessenberg pair is an ordered pair of diagonalizable linear maps on a
nonzero finite-dimensional vector space,
that each act on an eigenbasis of the other one in a Hessenberg fashion.
A Hessenberg system is a sequence
$\Phi = (A; \{E_i\}_{i=0}^d; A^*; \{E^*_i\}_{i=0}^d)$
consisting of a Hessenberg pair $A,A^*$ and an appropriate ordering $\{E_i\}_{i=0}^d$
(resp.\ $\{E^*_i\}_{i=0}^d$) of the primitive idempotents of $A$ (resp.\ $A^*$).
The ordering $\{E_i\}_{i=0}^d$ (resp.\ $\{E^*_i\}_{i=0}^d$) induces an ordering 
of the eignvalues of $A$ (resp.\ $A^*$) called
the eigenvalue sequence (resp.\ dual eigenvalue sequence) of $\Phi$.
It is known that $\Phi$ is determined up to isomorphism by its
parameter array;
this consists of the eigenvalue sequence of $\Phi$,
the dual eigenvalue sequence of $\Phi$,
and a sequence of nonzero scalars  $\{\phi_i\}_{i=1}^d$ called the split sequence of $\Phi$.
We are interested in some types of Hessenberg matrices,
said to be circular, quasi-circular, tridiagonal, and irreducible tridiagonal.
We are interested in the  families of  Hessenberg systems
for which  the associated Hessenberg matrices have
one of the above types.
A Hessenberg system of irreducible tridiagonal type is often called a Leonard system.
In this case the associated Hessenberg pair 
 satisfies  two relations,
called the tridiagonal relations.
We are interested in the family of Hessenberg systems for which
the associated Hessenberg pair satisfies the tridiagonal relations.
We are also interested in the family of Hessenberg systems for which
the eigenvalue sequence and dual eigenvalue sequence 
 satisfy
a linear three-term recurrence.
In the present paper we have two main goals.
First,
we show how the above families of Hessenberg systems are related to each other.
Second, we describe each family in terms of the
parameter array.
\end{abstract}

\section{Introduction}

This paper is about a linear-algebraic object called a Hessenberg pair.
This object was introduced by Godjali \cite{God} (where it was called a TH pair).
A Hessenberg pair is described as follows.
A square matrix is called  Hessenberg whenever each entry below the subdiagonal is zero and
each entry on the subdiagonal is nonzero.
For example, a $6 \times 6$  Hessenberg matrix has the form
\[
\text{Hessenberg:} \qquad
\begin{pmatrix}
* & *  & * & * & * & *  \\
\bullet & * & * & * & * & *  \\
 0 & \bullet &  * & * & * & *  \\
0 &  0 & \bullet &  * & * & * \\
0 & 0 &  0 & \bullet &  * & * \\
0 & 0 & 0 &  0 & \bullet &  *  \\
\end{pmatrix},\qquad
\]
where 
 $\bullet$ denotes an entry which must be nonzero,
and $*$ denotes an entry which might be nonzero.
A  Hessenberg pair is an ordered pair of diagonalizable linear maps on a
nonzero finite-dimensional vector space,
that each act on an eigenbasis of the other one in a
Hessenberg fashion.
A Hessenberg system is a sequence 
$\Phi  = (A; \{E_i\}_{i=0}^d; A^*; \{E^*_i\}_{i=0}^d)$
consisting of a
Hessenberg pair $A,A^*$ together with 
an appropriate ordering $\{E_i\}_{i=0}^d$ of the primitive idempotents of $A$ 
and 
an appropriate ordering $\{E^*_i\}_{i=0}^d$ of the primitive idempotents of $A^*$,
see Definition \ref{def:HS}.
The ordering $\{E_i\}_{i=0}^d$ (resp. $\{E^*_i\}_{i=0}^d$)
induces an ordering of the eigenvalues of $A$ (resp.\ $A^*$)
called the eigenvalue sequence (resp.\ dual eigenvalue sequence) of $\Phi$.
As explained in \cite{God},
the Hessenberg system $\Phi$
is determined up to isomorphism
by its parameter array; this consists of the eigenvalue sequence  of $\Phi$,
the dual eigenvalue sequence of  $\Phi $,
and a sequence of nonzero scalars $\{\phi_i\}_{i=1}^d$ called the split sequence of  $\Phi$.

In the present paper, we will discuss some special  Hessenberg pairs and systems,
for  which  the associated Hessenberg matrices have one of the
following forms.
\[
\begin{array}{ll}
\text{Circular Hessenberg:}
&
\begin{pmatrix}
* & *  & 0 & 0 & 0 & \bullet \\
\bullet & * & * & 0 & 0 & 0  \\
 0 & \bullet &  * & * & 0 & 0  \\
0 &  0 & \bullet &  * & * & 0 \\
0 & 0 &  0 & \bullet &  * & * \\
0 & 0 & 0 &  0 & \bullet &  *  \\
\end{pmatrix},
\\
\text{Quasi-Circular Hessenberg:}  \rule{0mm}{10ex}
&
\begin{pmatrix}
* & *  & 0 & 0 & 0 & *  \\
\bullet & * & * & 0 & 0 & 0  \\
 0 & \bullet &  * & * & 0 & 0  \\
0 &  0 & \bullet &  * & * & 0 \\
0 & 0 &  0 & \bullet &  * & * \\
0 & 0 & 0 &  0 & \bullet &  *  \\
\end{pmatrix},  
\\ 
\text{Tridiagonal Hessenberg:}   \rule{0mm}{10ex}
&
\begin{pmatrix}
* & *  & 0 & 0 & 0 & 0  \\
\bullet & * & * & 0 & 0 & 0  \\
 0 & \bullet &  * & * & 0 & 0  \\
0 &  0 & \bullet &  * & * & 0 \\
0 & 0 &  0 & \bullet &  * & * \\
0 & 0 & 0 &  0 & \bullet &  *  \\
\end{pmatrix},
\\
\text{Irreducible Tridiagonal:}  \rule{0mm}{10ex}
&
\begin{pmatrix}
* & \bullet  & 0 & 0 & 0 & 0  \\
\bullet & * & \bullet & 0 & 0 &  0   \\
 0 & \bullet &  * & \bullet & 0 & 0  \\
0 &  0 & \bullet &  * & \bullet & 0 \\
0 & 0 &  0 & \bullet &  * & \bullet \\
0 & 0 & 0 &  0 & \bullet &  *  \\
\end{pmatrix}.
\end{array}
\]
A Hessenberg pair (resp.\ system) of Irreducible Tridiagonal type is often called a Leonard pair
(resp.\ Leonard system), see \cite{T:Leonard}.
We abbreviate
\begin{center}
\begin{tabular}{ll}
HS: & Hessenberg system,
\\
CHS: & Circular Hessenberg system,
\\
QCHS: & Quasi-Circular Hessenberg system,
\\
THS: & Tridiagonal Hessenberg system,
\\
LS: & Leonard system.
\end{tabular}
\end{center}
In the present paper,
we will also discuss the following two families of Hessenberg pairs and systems.
By  \cite[Theorem 1.12]{T:Leonard}, for a Leonard pair $A,A^*$
there exist scalars 
$\beta, \gamma, \gamma^*, \varrho, \varrho^*$ such that
\begin{align}
 0 &= [A, A^2 A^* - \beta A A^* A + A^* A^2 - \gamma (A A^* + A^* A) - \varrho A^*],   \label{eq:td1a}
\\
 0 &= [A^*, A^{*2} A - \beta A^* A A^* + A A^{*2} - \gamma^* (A^* A + A A^*) - \varrho^* A],  \label{eq:td2a}
\end{align}
where  $[X,Y] = XY - YX$.
The relations \eqref{eq:td1a}, \eqref{eq:td2a} are called the tridiagonal relations (TD relations).
We say that a Hessenberg pair $A,A^*$ is TD whenever 
there exist scalars 
$\beta, \gamma, \gamma^*, \varrho, \varrho^*$ that satisfy
\eqref{eq:td1a}, \eqref{eq:td2a}.
A Hessenberg system is called TD whenever the corresponding Hessenberg pair is TD.
By \cite[Lemma  12.7]{T:Leonard}, for a Leonard system $\Phi$ the eigenvalue sequence
and the dual eigenvalue sequence satisfy a linear three-term recurrence.
We say that a Hessenberg system $\Phi$ is REC whenever the
eigenvalue sequence and dual eigenvalue sequence  of $\Phi$ satisfy this recurrence.

We have some historical comments about the above families  of Hessenberg systems.
The CHS family was introduced by Jae-Ho Lee \cite{JHL}.
The QCHS family and the THS family are introduced in this paper.
The Leonard systems  were introduced in \cite{T:Leonard}.
In \cite{T:survey, T:qRacah, T:notes, T:parray}
it is described how the Leonard systems correspond to the orthogonal polynomials from
the terminating branch of the Askey scheme.
This branch    consists of the $q$-Racah polynomials and their relatives, see \cite{Koe}.
Leonard systems appear in the context of $Q$-polynomial distance-regular graphs,
where they are used to describe the thin irreducible modules for the subconstituent algebra,
see 
\cite{Cau, CMT, Cerzo, NT:LPspin}.
In  \cite{CT}, Leonard systems are used to resolve the Kresch-Tamvakis conjecture involving
$_4 F_3$ hypergeometric series.
In \cite{NT:DAHA}, Leonard systems are used to describe the finite-dimensional irreducible  modules
for the rank one DAHA of type $(C_1^\vee, C_1)$.
For additional results about Leonard systems, 
see \cite{T:TDD, NT:switch,T:decomp, T:24points, NT:Krawt, NT:trid}.
The TD relations first appeared in \cite[Lemma 5.4]{T:subconst3}.
The TD relations are the defining relations for the tridiagonal algebra, see \cite[Definition 3.9]{T:TDrel}.

In the present paper we have two main goals.
Our first goal is to describe how the following families of Hessenberg systems are related:
\begin{equation}
\text{HS, \quad CHS, \quad QCHS, \quad THS, \quad LS, \quad TD, \quad REC}.    \label{eq:lst}
\end{equation}
Our finding is summarized by the following diagrams,
which show the logical implications:
\begin{equation}
\text{QCHS} \quad   \Leftrightarrow  \quad 
\text{TD}   \label{eq:diag1} 
\end{equation}
\begin{equation}
\text{LS} \quad \Rightarrow \quad 
\text{THS} \quad  \Rightarrow \quad 
\text{TD} \quad \Rightarrow \quad  
\text{REC} \quad  \Rightarrow  \quad   
\text{HS}                          \label{eq:diag2}
\end{equation}
Moreover,
CHS is the compliment of THS in TD.

Our second goal is  to consider each family of Hessenberg system in \eqref{eq:diag2},
and describe it in terms of the parameter array.
Some of these descriptions are cited from the prior
literature,
and some are new to the paper.
In the table below,
we state where in the paper each description is located:
\begin{center}
\begin{tabular}{c|c}
family & description in terms of the parameter array
\\ \hline
LS & Lemma \ref{lem:LSclassify}        \rule{0mm}{3.5ex}
\\
THS & Propositions \ref{prop:QCHTD},     \ref{prop:vthd2}  \rule{0mm}{3.5ex}
\\
TD   &  Propositions \ref{prop:vth}, \ref{prop:TDd2}       \rule{0mm}{3.5ex}
\\
REC   & Definition \ref{def:recH}     \rule{0mm}{3.5ex}
\\
HS & Lemma \ref{lem:Hclassify}            \rule{0mm}{3.5ex}
\end{tabular}
\end{center}

The paper is organized as follows.
Section \ref{sec:preli} contains some preliminaries.
In Section \ref{sec:rec} we recall some results about recurrent sequences.
In Section \ref{sec:matrix} we define some  special Hessenberg matrices.
In Section \ref{sec:HS} we recall some results concerning a Hessenberg  system.
In Section \ref{sec:CHS} we recall some results concerning CHS.
In Section \ref{sec:QCHS} we describe QCHS.
In Section \ref{sec:LS} we recall some results concerning  LS.
In Section \ref{sec:THS} we describe THS.
In Sections  \ref{sec:TD}, \ref{sec:adjust} we recall the TD relations and describe the TD property.
In Section \ref{sec:triple} we obtain some necessary and sufficient conditions for certain
triple products to vanish.
We use these results to describe TD and THS.
In Section \ref{sec:proof} we complete the proofs of \eqref{eq:diag1}, \eqref{eq:diag2}
and prove that CHS is the compliment of THS in TD.

.

\section{Preliminaries}
\label{sec:preli}
\ifDRAFT {\rm sec:preli}. \fi

Throughout the paper, 
we use the following concepts and notations.
Fix an integer $d \geq 2$.
Let $\F$ denote a  field, and let
$V$ denote a vector space over $\F$ with dimension $d+1$.
Let $\text{\rm End}(V)$ denote the $\F$-algebra consisting of the  $\F$-linear maps $V \to V$.
Let $I \in \text{\rm End}(V)$ denote the identity map.
For $X, Y \in \text{End}(V)$, define $[X,Y] = X Y - Y X$.
An element $A \in \text{End}(V)$ is said to be {\em diagonalizable} whenever $V$ is spanned by
the eigenspaces of $A$.
An element $A \in \text{End}(V)$ is said to be {\em multiplicity-free} 
whenever $A$ is diagonalizable and each eigenspace of $A$  has dimension one.
Assume that $A$ is multiplicity-free,  and
let $\{V_i\}_{i=0}^d$ denote an ordering of the eigenspaces of $A$.
The sum $V = \sum_{i=0}^d V_i$ is direct.
For $0 \leq i \leq d$ let $\th_i$ denote the eigenvalue of $A$ for $V_i$.
The scalars  $\{\th_i\}_{i=0}^d$ are contained in $\F$ and mutually distinct.
For $0 \leq i \leq d$ define $E_i \in \text{End}(V)$ such that
$(E_i - I) V_i = 0$ and 
$E_i V_j = 0$ if $j \neq i$ $(0 \leq j \leq d)$.
Then
(i) $ A E_i = E_i A = \th_i E_i$ $(0 \leq i \leq d)$;
(ii) $E_i E_j = \delta_{i,j}E_i$ $(0 \leq i,j \leq d)$;
(iii) $I = \sum_{i=0}^d E_i $;
(iv) $A =\sum_{i=0}^d \th_i E_i$;
(v) $V_i = E_i V$ $(0 \leq i  \leq d)$.
We call  $E_i$ the {\em primitive idempotent of $A$ associated with $\th_i$} $(0 \leq i \leq d)$.
By linear algebra,
\[
E_i = \prod_{\text{\scriptsize $\begin{matrix}  0 \leq j \leq d \\ j \neq i \end{matrix}$} }
           \frac{A-\th_j I} { \th_i - \th_j}
\qquad\qquad ( 0 \leq i \leq d).    
\]
The elements $\{E_i\}_{i=0}^d $ form a basis for  the subalgebra of $\text{End}(V)$ generated by $A$.
Let $\Mat_{d+1}(\F)$ denote the $\F$-algebra consisting of
the $(d+1) \times (d+1)$ matrices that have all entry in $\F$.
We index the rows and columns by $0,1, \ldots, d$.

\section{Recurrent sequences}
\label{sec:rec}
\ifDRAFT {\rm sec:rec}. \fi

In this section,
we recall from \cite{T:Leonard} and \cite{T:notes}  some results about recurrent sequences.

\begin{defi}  {\rm (See \cite[Definition 8.2]{T:Leonard}.) }
\label{def:betarec}    \samepage
\ifDRAFT {\rm def:betarec}. \fi
Let $\beta, \gamma, \varrho$ denote scalars in   $\F$,
and let $\{\th_i\}_{i=0}^d$ denote a sequence of scalars   in $\F$.
\begin{itemize}
\item[\rm (i)]
The sequence  $\{\th_i\}_{i=0}^d$ is said to be {\em recurrent} whenever
$\th_{i-1} \neq \th_i$ for $2 \leq i \leq d-1$ and
\begin{equation}
\frac{\th_{i-2} - \th_{i+1} } { \th_{i-1} - \th_i }            \label{eq:indep0}
\end{equation}
is independent of $i$ for $2 \leq i \leq d-1$.
\item[\rm (ii)]
The sequence  $\{\th_i\}_{i=0}^d$ is said to be {\em $\beta$-recurrent} whenever
\[
\th_{i-2} - (\beta+1) \th_{i-1}  + (\beta+1) \th_i - \th_{i+1} = 0
\qquad \qquad
(2 \leq i \leq d-1).
\]
\item[\rm (iii)]
The sequence  $\{\th_i\}_{i=0}^d$ is said to be {\em $(\beta,\gamma)$-recurrent} whenever
\[
\gamma = \th_{i-1} - \beta \th_i + \th_{i+1}
\qquad\qquad   (1 \leq i \leq d-1).
\]
\item[\rm (iv)]
The sequence  $\{\th_i\}_{i=0}^d$ is said to be {\em $(\beta,\gamma, \varrho)$-recurrent} whenever
\[
\varrho = \th_{i-1}^2 - \beta \th_{i-1} \th_i + \th_i ^2 - \gamma (\th_{i-1} + \th_i)
\qquad\qquad (1 \leq i \leq d).
\]
\end{itemize}
\end{defi}

\begin{note}   \label{note:rec}   \samepage
\ifDRAFT {\rm note:rec}. \fi
Referring to Definition \ref{def:betarec}, assume that $d=2$.
Then $\{\th_i\}_{i=0}^d$ is $\beta$-recurrent for all $\beta \in \F$.
\end{note}

\begin{lemma}   \label{lem:recinv}   \samepage
\ifDRAFT {\rm lem:recinv}. \fi
Referring to Definition \ref{def:betarec},
$\{\th_i\}_{i=0}^d$ is recurrent (resp.\  $\beta$-recurrent) (resp.\  $(\beta,\gamma)$-recurrent)
(resp. $(\beta,\gamma, \varrho)$-recurrent)
if and only if  $\{\th_{d-i}\}_{i=0}^d$ is recurrent
(resp.\  $\beta$-recurrent)  (resp.\  $(\beta,\gamma)$-recurrent)
(resp. $(\beta,\gamma, \varrho)$-recurrent).
\end{lemma}

\begin{proof}
By Definition \ref{def:betarec}.
\end{proof}

\begin{lemma}    {\rm (See \cite[Lemma 8.3]{T:Leonard}.) }
\label{lem:recseq0}    \samepage
\ifDRAFT {\rm lem:redcseq0}. \fi
Let $\{\th_i\}_{i=0}^d$  denote a  sequence of scalars in $\F$.
Then the following are equivalent:
\begin{itemize}
\item[\rm (i)]
 $\{\th_i\}_{i=0}^d$ is recurrent;
\item[\rm (ii)]
there exists $\beta \in \F$ such that $\{\th_i\}_{i=0}^d$ is $\beta$-recurrent,
and $\th_{i-1} \neq \th_i$ for $2 \leq i  \leq d-1$.
\end{itemize}
Suppose {\rm (i), (ii)} hold and $d \geq 3$. 
Then the common value of \eqref{eq:indep0} is $\beta+1$.
\end{lemma}

\begin{lemma}    {\rm (See \cite[Lemma 8.4]{T:Leonard}.) }
\label{lem:recseq1}    \samepage
\ifDRAFT {\rm lem:redcseq1}. \fi
Let $\{\th_i\}_{i=0}^d$  denote a  sequence of scalars in $\F$.
Then for $\beta \in \F$
the following are equivalent:
\begin{itemize}
\item[\rm (i)]
 $\{\th_i\}_{i=0}^d$ is $\beta$-recurrent;
\item[\rm (ii)]
there exists $\gamma \in \F$ such that $\{\th_i\}_{i=0}^d$ is $(\beta, \gamma)$-recurrent.
\end{itemize}
\end{lemma}

\begin{lemma}    {\rm (See \cite[Lemma 8.5]{T:Leonard}.) }
\label{lem:recseq2}    \samepage
\ifDRAFT {\rm lem:redseq2}. \fi
Let $\{\th_i\}_{i=0}^d$  denote a  sequence of scalars in $\F$.
Then for $\beta, \gamma \in \F$
the following hold.
\begin{itemize}
\item[\rm (i)]
Suppose $\{\th_i\}_{i=0}^d$  is  $(\beta, \gamma)$-recurrent.
Then there exists $\varrho \in \F$ such that $\{\th_i\}_{i=0}^d$  is $(\beta, \gamma, \varrho)$-recurrent.
\item[\rm (ii)]
Suppose $\th_{i-1} \neq \th_{i+1}$ $(1 \leq i \leq d-1)$, and
there exists  $\varrho \in \F$ such that  $\{\th_i\}_{i=0}^d$  is $(\beta, \gamma, \varrho)$-recurrent.
Then $\{\th_i\}_{i=0}^d$  is $(\beta, \gamma)$-recurrent.
\end{itemize}
\end{lemma}

\begin{lemma}   {\rm (See \cite[Proposition 13.4]{T:notes}.)}
\label{lem:Tvth}     \samepage
\ifDRAFT {\rm lem:Tvth}. \fi
Let $\{\th_i\}_{i=0}^d$ denote a sequence of mutually distinct scalars in $\F$.
Assume that there exists $\beta \in \F$ such that  $\{\th_i\}_{i=0}^d$
is $\beta$-recurrent.
Then for scalars $\{\vth_i\}_{i=0}^{d+1}$ in $\F$ the following are
equivalent:
\begin{itemize}
\item[\rm (i)]
$\displaystyle
 \vth_i = \vth_1 \sum_{\ell=0}^{i-1} \frac{\th_\ell - \th_{d-\ell}} { \th_0 - \th_d}
 \qquad\qquad (0 \leq i \leq d+1)$;
\item[\rm (ii)]
the sequence $\{\vth_i\}_{i=0}^{d+1}$ is $\beta$-recurrent and
\[
\vth_0 = 0,
\qquad\qquad
\vth_1 = \vth_d,
\qquad\qquad
\vth_{d+1} = 0.
\]
\end{itemize}
\end{lemma}

\section{Some special Hessenberg matrices}
\label{sec:matrix}
\ifDRAFT {\rm sec:matrix}. \fi

In this section, we describe five types of  matrices that will play a role in
our main results.

\begin{defi}   \label{def:Hessen}    \samepage
\ifDRAFT {\rm def:Hessrn}. \fi
A matrix  $M \in \Mat_{d+1}(\F)$
is said to be  {\em Hessenberg} (or {\em H}) whenever
for $0 \leq i,j \leq d$,
\begin{itemize}
\item[\rm (i)]
$M_{i,j} = 0$ $\;$ if  $\;$ $1 < i-j$;
\item[(ii)]
$M_{i,j} \neq 0$ $\;$  if $\;$  $1= i-j$.
\end{itemize}
\end{defi}

\begin{defi}   \label{def:CH}    \samepage
\ifDRAFT {\rm def:CH}. \fi
A matrix  $M \in  \Mat_{d+1}(\F)$
is said to be  {\em Circular Hessenberg} (or {\em CH}) whenever
for $0 \leq i,j \leq d$,
\begin{itemize}
\item[\rm (i)]
$M_{i,j} = 0$ $\;$  if $\;$ $1 < i-j$ $\;$ or  $\:$ $1 < j-i< d$;
\item[(ii)]
$M_{i,j} \neq 0$ $\;$   if $\;$  $1= i-j$ $\;$ or $\;$ $j-i=d$.
\end{itemize}
\end{defi}

\begin{defi}   \label{def:QCH}    \samepage
\ifDRAFT {\rm def:QCH}. \fi
A matrix  $M \in \Mat_{d+1}(\F)$
is said to be  {\em Quasi-Circular Hessenberg} (or {\em QCH}) whenever
for $0 \leq i,j \leq d$,
\begin{itemize}
\item[\rm (i)]
$M_{i,j} = 0$ $\;$  if $\;$ $1 < i-j$ $\;$ or  $\:$ $1 < j-i< d$;
\item[(ii)]
$M_{i,j} \neq 0$ $\;$   if $\;$  $1= i-j$.
\end{itemize}
\end{defi}

\begin{defi}   \label{def:TH}    \samepage
\ifDRAFT {\rm def:TH}. \fi
A matrix  $M \in  \Mat_{d+1}(\F)$
is said to be  {\em Tridiagonal  Hessenberg} (or {\em TH}) whenever
for $0 \leq i,j \leq d$,
\begin{itemize}
\item[\rm (i)]
$M_{i,j} = 0$ $\;$  if $\;$ $1 < |i-j|$;
\item[(ii)]
$M_{i,j} \neq 0$ $\;$   if $\;$  $1= i-j$.
\end{itemize}
\end{defi}

\begin{defi}   \label{def:IT}    \samepage
\ifDRAFT {\rm def:IT}. \fi
A matrix  $M \in \Mat_{d+1}(\F)$
is said to be  {\em Irreducible Tridiagonal} (or {\em IT}) whenever
for $0 \leq i,j \leq d$,
\begin{itemize}
\item[\rm (i)]
$M_{i,j} = 0$ $\;$  if $\;$ $1 < |i-j|$;
\item[(ii)]
$M_{i,j} \neq 0$ $\;$   if $\;$  $1= |i-j|$.
\end{itemize}
\end{defi}

We just described five types of  matrices.
We now illustrate each type  for $d=5$.
In the illustration,
we denote by $\bullet$ an entry which must be nonzero,
and by $*$ an entry which might be nonzero.

\begin{align*}
\text{H:} &
\begin{pmatrix}
* & *  & * & * & * & *  \\
\bullet & * & * & * & * & *  \\
 0 & \bullet &  * & * & * & *  \\
0 &  0 & \bullet &  * & * & * \\
0 & 0 &  0 & \bullet &  * & * \\
0 & 0 & 0 &  0 & \bullet &  *  \\
\end{pmatrix},
&
\text{CH:} &
\begin{pmatrix}
* & *  & 0 & 0 & 0 & \bullet \\
\bullet & * & * & 0 & 0 & 0  \\
 0 & \bullet &  * & * & 0 & 0  \\
0 &  0 & \bullet &  * & * & 0 \\
0 & 0 &  0 & \bullet &  * & * \\
0 & 0 & 0 &  0 & \bullet &  *  \\
\end{pmatrix},
&
\text{QCH:} &
\begin{pmatrix}
* & *  & 0 & 0 & 0 & *  \\
\bullet & * & * & 0 & 0 & 0  \\
 0 & \bullet &  * & * & 0 & 0  \\
0 &  0 & \bullet &  * & * & 0 \\
0 & 0 &  0 & \bullet &  * & * \\
0 & 0 & 0 &  0 & \bullet &  *  \\
\end{pmatrix},
\\\\
\text{TH:} &
\begin{pmatrix}
* & *  & 0 & 0 & 0 & 0  \\
\bullet & * & * & 0 & 0 & 0  \\
 0 & \bullet &  * & * & 0 & 0  \\
0 &  0 & \bullet &  * & * & 0 \\
0 & 0 &  0 & \bullet &  * & * \\
0 & 0 & 0 &  0 & \bullet &  *  \\
\end{pmatrix},
&
\text{IT:} &
\begin{pmatrix}
* & \bullet  & 0 & 0 & 0 & 0  \\
\bullet & * & \bullet & 0 & 0 &  0   \\
 0 & \bullet &  * & \bullet & 0 & 0  \\
0 &  0 & \bullet &  * & \bullet & 0 \\
0 & 0 &  0 & \bullet &  * & \bullet \\
0 & 0 & 0 &  0 & \bullet &  *  \\
\end{pmatrix}.
\end{align*}

We comment on how the above types of matrices are related.
The following diagram shows the logical implications:
\begin{center}
\begin{tabular}{cccccccc}
& &   &  & CH
\\
& & & & $\Downarrow$
\\
IT
&
$\Rightarrow$ 
&
TH
&
$\Rightarrow$
& 
QCH
&
$\Rightarrow$ 
&
H
\end{tabular}
\end{center}

\section{Hessenberg pairs and systems}
\label{sec:HS}
\ifDRAFT {\rm sec:HS}. \fi

In this section, we recall from \cite{God} the notion of a Hessenberg pair and a Hessenberg system.

\begin{defi}   {\rm (See \cite[Definition 1.1]{God}.) }
\label{def:HP} \samepage
\ifDRAFT {\rm ded:HP} \fi
By a {\em Hessenberg pair} on $V$,
we mean an ordered pair $A,A^*$ of elements in $\text{End}(V)$
that satisfy the following {\rm (i), (ii)}:
\begin{itemize}
\item[\rm (i)]
there exists a basis for $V$ with respect to which the matrix representing $A$ is Hessenberg
and the matrix representing $A^*$ is diagonal;
\item[\rm (ii)]
there exists a basis for $V$ with respect to which the matrix representing $A^*$ is Hessenberg
and the matrix representing $A$ is diagonal.
\end{itemize}
\end{defi}

\begin{note} \label{note:diff}  \samepage
\ifDRAFT {\rm not3:diff}. \fi
Our concept of a Hessenberg pair is slightly different from the one in \cite{God}.
What we call a Hessenberg pair is called a TH pair in \cite{God}.
\end{note}

\begin{lemma}   {\rm (See \cite[Lemma 2.1]{God}.) }
\label{lem:mfree}    \samepage
\ifDRAFT {\rm lem:mfee}. \fi
Let $A, A^*$ denote a Hessenberg pair on $V$.
Then each of $A,A^*$ is multiplicity-free.
\end{lemma}

\begin{defi}   {\rm (See \cite[Definition 2.2]{God}.) }
\label{def:HS}    \samepage
\ifDRAFT {\rm def:HS}. \fi
By a {\em Hessenberg system} (or {\em HS}) on $V$
we mean a sequence
\[
\Phi  = (A; \{E_i\}_{i=0}^d; A^*; \{E^*_i\}_{i=0}^d)
\]
 of elements of $\text{End}(V)$ that satisfy the following (i)--(v):
\begin{itemize}
\item[(i)]
each of $A,A^*$ is multiplicity-free;
\item[(ii)]
$\{E_i\}_{i=0}^d$ is an ordering of the primitive idempotents of $A$;
\item[(iii)]
$\{E^*_i\}_{i=0}^d$ is an ordering of the primitive idempotents of $A^*$;
\item[(iv)]
$\displaystyle
E_i A^* E_j =
\begin{cases}
 0 & \text{ if $1 < i-j$},
\\
\neq 0 & \text{ if $1 = i-j$}
\end{cases}
\qquad\qquad ( 0 \leq i,j \leq d);
$
\item[(v)]
$\displaystyle
E^*_i A E^*_j =
\begin{cases}
 0 & \text{ if $1 < i-j$},
\\
\neq 0 & \text{ if $1= i-j$}
\end{cases}
\qquad\qquad ( 0 \leq i,j \leq d).
$
\end{itemize}
We say that $\Phi$ is {\em over $\F$}.
\end{defi}

Hessenberg  pairs and Hessenberg systems are related as follows.
Let  
\[
(A; \{E_i\}_{i=0}^d; A^*; \{E^*_i\}_{i=0}^d)
\]
denote a Hessenberg  system on $V$.
Then $A,A^*$ is a Hessenberg  pair on $V$.
Conversely,
let $A,A^*$ denote a Hessenberg  pair on $V$.
Let $\{v_i\}_{i=0}^d$ (resp.\ $\{v^*_i\}_{i=0}^d$) denote a basis for $V$
that satisfies Definition \ref{def:HP}(ii) (resp.\ Definition \ref{def:HP}(i)).
For $0 \leq i \leq d$,
the vector $v_i$ (resp.\ $v^*_i$) is an eigenvector for $A$ (resp.\ $A^*$);
let $E_i$  (resp.\ $E^*_i$) denote the corresponding primitive idempotent of $A$ (resp.\ $A^*$).
Then   $(A; \{E_i\}_{i=0}^d; A^*; \{E^*_i\}_{i=0}^d)$
is a Hessenberg system on $V$.

\begin{defi}    \label{def:eigenseq}    \samepage
\ifDRAFT  {\rm def:eigenseq}. \fi
Referring to Definition \ref{def:HS},
for $0 \leq i \leq d$ let $\th_i$ (resp.\ $\th^*_i$) denote the eigenvalue of $A$ (resp.\ $A^*$)
for $E_i$ (resp.\ $E^*_i$).
We call $\{\th_i\}_{i=0}^d$ (resp. $\{\th^*_i\}_{i=0}^d$) the {\em eigenvalue sequence}
(resp.\ {\em dual eigenvalue sequence}) of  $\Phi$.
\end{defi}

We recall the concept of isomorphism for Hessenberg  systems.
Let
\[
\Phi = (A; \{E_i\}_{i=0}^d; A^*; \{E^*_i\}_{i=0}^d)
\]
denote a Hessenberg system on $V$.
Let $V'$ denote a vector space over $\F$ with dimension $d+1$.
For an algebra isomorphism $\sigma :  \text{End}(V) \to \text{End}(V')$,
define 
\[
\Phi^\sigma = (A^\sigma; \{(E_i)^\sigma \}_{i=0}^d; (A^*)^\sigma ; \{(E^*_i)^\sigma\}_{i=0}^d).
\]
Then $\Phi^\sigma$ is a Hessenberg system on $V'$.
Let $\Phi'$ denote a Hessenberg system  on $V'$.
By an {\em isomorphism of Hessenberg  systems from $\Phi$ to $\Phi'$} we mean
an algebra isomorphism $\sigma : \text{End}(V) \to \text{End}(V')$ such that
$\Phi^\sigma = \Phi'$.
We say that $\Phi$ and $\Phi'$ are {\em isomorphic} whenever there exists an
isomorphism of Hessenberg systems from $\Phi$ to $\Phi'$.

\begin{defi}   {\rm (See \cite[Definition 2.6]{God}.) }
\label{def:dual}    \samepage
\ifDRAFT {\rm def:dual}. \fi
Let $\Phi = (A; \{E_i\}_{i=0}^d; A^*; \{E^*_i\}_{i=0}^d)$
denote a Hessenberg system  on $V$.
Then
$\Phi^* = (A^*; \{E^*_i\}_{i=0}^d; A;  \{E_i\}_{i=0}^d)$
is a Hessenberg system  on $V$.
We call $\Phi^*$ the {\em dual} of $\Phi$.
For an object $f$ attached to $\Phi$, let $f^*$ denote the corresponding 
object attached to $\Phi^*$.
\end{defi}

Let $\Phi = (A; \{E_i\}_{i=0}^d; A^*; \{E^*_i\}_{i=0}^d)$
denote a Hessenberg system  on $V$.
We recall the $\Phi$-split decomposition of $V$.
For $0 \leq i \leq d$ define
\begin{equation}
U_i = (E^*_0  V + E^*_1 V + \cdots + E^*_i V) \cap
 (E_0  V + E_1 V + \cdots + E_{d-i} V).                       \label{eq:Ui}
\end{equation}
By \cite[Lemma 2.4, Theorem 4.1]{T:decomp} we have the following facts.
The dimension of $U_i$ is $1$ for $0 \leq i  \leq d$.
The sum $V = \sum_{i=0}^d U_i$ is direct.
Moreover
\begin{align*}
 (A - \th_{d-i} I ) U_i &= U_{i+1} \quad ( 0 \leq i \leq d-1),
&
 (A-\th_0 I) U_d &= 0,
\\
 ( A^* - \th^*_i I) U_i &= U_{i-1} \quad (1 \leq i \leq d),
&
 (A^* -\th^*_0 I) U_0 &= 0.
\end{align*}
By \eqref{eq:Ui} we have $U_d = E_0 V$.
Thus
\[
U_i  = (A^* - \th^*_{i+1} I) \cdots (A^* - \th^*_{d-1} I) ( A^* - \th^*_d I) E_0 V
\qquad\qquad (0 \leq i \leq d).
\]
Pick $0 \neq \eta_0 \in E_0 V$.
Then for $0 \leq i \leq d$ the vector 
\[
 (A^* - \th^*_{i+1} I) \cdots (A^* - \th^*_{d-1} I) ( A^* - \th^*_d I) \eta_0
\]
is a basis for $U_i$.
The sequence $\{U_i\}_{i=0}^d$ is called the {\em $\Phi$-split decomposition of $V$}.

\begin{defi}    {\rm (See \cite[Definition 4.2]{God}.)}
\label{def:xplitbasis}    \samepage
\ifDRAFT {\rm def:splitbaasis}. \fi
A basis for $V$ is said to be {\em $\Phi$-split}
whenever
 it has the form
\[
 (A^* - \th^*_{i+1} I) \cdots (A^* - \th^*_{d-1} I) ( A^* - \th^*_d I) \eta_0    
\qquad\qquad (0 \leq i \leq d),
\]
where  $0 \neq \eta_0 \in E_0 V$.
\end{defi}

Pick $1 \leq i \leq d$.
By construction,
\[
(A-\th_{d-i+1} I) (A^* - \th^*_i I) U_i = U_i.
\]
So, there exists a unique nonzero  $\phi_i \in \F$ such that
the following holds on $U_i$:
\[
(A-\th_{d-i+1} I) (A^* - \th^*_i I)  = \phi_i  I. 
\]
By \cite[Lemma 4.6]{God} the following holds on $U_{i-1}$:
\[
 (A^* - \th^*_i I)(A-\th_{d-i+1} I)  = \phi_i  I.  
\]
We call the sequence $\{\phi_i\}_{i=1}^d$ the {\em split sequence of $\Phi$}.
For notational convenience,  define
\[
\phi_0 = 0,
\qquad\qquad
\phi_{d+1} = 0.   
\]

\begin{defi}    \label{def:parray}    \samepage
\ifDRAFT {\rm def:parray}. \fi
Let $\Phi$ denote a Hessenberg system on $V$.
By the {\em paramater array} of $\Phi$ we mean the sequence
\[
 (\{\th_i\}_{i=0}^d; \{\th^*_i\}_{i=0}^d; \{\phi_i\}_{i=1}^d),  
\]
where $\{\th_i\}_{i=0}^d$ (resp.$\{\th^*_i\}_{i=0}^d$) is the
eigenvalue sequence (resp.\ dual eigenvalue sequence) of $\Phi$,
and $\{\phi_i\}_{i=1}^d$ is the split sequence of $\Phi$.
\end{defi}

\begin{lemma}   {\rm (See \cite[Lemma 5.1]{God}.) }
\label{lem:Phis}    \samepage
\ifDRAFT {\rm lem:Phis}. \fi
Referring to Definition \ref{def:parray},
the parameter array of $\Phi^*$ is
\[
(\{\th^*_i\}_{i=0}^d; \{\th_i\}_{i=0}^d; \{\phi_{d-i+1}\}_{i=1}^d).
\]
\end{lemma}

Let $\{v_i\}_{i=0}^d$ denote a basis for $V$.
By the {\em inversion} of this basis we mean the basis $\{v_{d-i}\}_{i=0}^d$.

\begin{lemma}     {\rm (See \cite[Propositions 4.4, 5,3, 5.6, 5.9]{God}.) }
\label{lem:matrixAAs}
\ifDRAFT {\rm lem:matrixAAs}. \fi
Let $\Phi = (A; \{E_i\}_{i=0}^d; A^*; \{E^*_i\}_{i=0}^d)$
denote a Hessenberg system on $V$, with parameter array
$(\{\th_i\}_{i=0}^d;$ $\{\th^*_i\}_{i=0}^d; \{\phi_i\}_{i=1}^d)$.
Then the following hold.
\begin{itemize}
\item[\rm (i)]
With respect to a $\Phi$-split  basis for $V$,
the matrices representing  $A,A^*$ are
\[
A :
\begin{pmatrix}
\th_d & & & & & \text{\bf 0}   \\
\phi_1 & \th_{d-1} \\
  & \phi_2 &\th_{d-2} \\
  &    & \cdot & \cdot \\
  &    &     &  \cdot & \cdot \\
\text{\bf 0} & & & & \phi_d & \th_0
\end{pmatrix},
\qquad\qquad
A^* :
\begin{pmatrix}
\th^*_0 & 1 & & & & \text{\bf 0}   \\
 & \th^*_1& 1 \\
  &  &\th^*_2 &\cdot \\
  &    &  & \cdot & \cdot\\
  &    &   &  &  \cdot & 1 \\
\text{\bf 0} & & & &  & \th^*_d
\end{pmatrix}.    
\]
\item[\rm (ii)]
With respect to a $\Phi^*$-split  basis for $V$,
the matrices representing  $A,A^*$ are
\[
A :
\begin{pmatrix}
\th_0 & 1 & & & & \text{\bf 0}   \\
 & \th_1& 1 \\
  &  &\th_2 &\cdot \\
  &    &  & \cdot & \cdot\\
  &    &   &  &  \cdot & 1 \\
\text{\bf 0} & & & &  & \th_d
\end{pmatrix},
\qquad\qquad        
A^* :
\begin{pmatrix}
\th^*_d & & & & & \text{\bf 0}   \\
\phi_d & \th^*_{d-1} \\
  & \phi_{d-1} &\th^*_{d-2} \\
  &    & \cdot & \cdot \\
  &    &     &  \cdot & \cdot \\
\text{\bf 0} & & & & \phi_1 & \th^*_0
\end{pmatrix}.   
\]
\item[\rm (iii)]
With respect to an inverted $\Phi$-split basis for $V$,
the matrices representing  $A,A^*$ are
\[
A :
\begin{pmatrix}
\th_0 & \phi_d & & & & \text{\bf 0}   \\
 & \th_1& \phi_{d-1} \\
  &  &\th_2 &\cdot \\
  &    &  & \cdot & \cdot\\
  &    &   &  &  \cdot & \phi_1 \\
\text{\bf 0} & & & &  & \th_d
\end{pmatrix},
\qquad\qquad        
A^* :
\begin{pmatrix}
\th^*_d & & & & & \text{\bf 0}   \\
1 & \th^*_{d-1} \\
  & 1 &\th^*_{d-2} \\
  &    & \cdot & \cdot \\
  &    &     &  \cdot & \cdot \\
\text{\bf 0} & & & & 1 & \th^*_0
\end{pmatrix}.
\]
\item[\rm (iv)]
With respect to an inverted $\Phi^*$-split  basis for $V$,
the matrices representing  $A,A^*$ are
\[
A :
\begin{pmatrix}
\th_d & & & & & \text{\bf 0}   \\
1 & \th_{d-1} \\
  & 1 &\th_{d-2} \\
  &    & \cdot & \cdot \\
  &    &     &  \cdot & \cdot \\
\text{\bf 0} & & & & 1 & \th_0
\end{pmatrix},
\qquad\qquad
A^* :
\begin{pmatrix}
\th^*_0 & \phi_1 & & & & \text{\bf 0}   \\
 & \th^*_1& \phi_2 \\
  &  &\th^*_2 &\cdot \\
  &    &  & \cdot & \cdot\\
  &    &   &  &  \cdot & \phi_d \\
\text{\bf 0} & & & &  & \th^*_d
\end{pmatrix}.    
\]
\end{itemize}
\end{lemma}

\begin{lemma}   {\rm (See \cite[Lemma 6.1, Theorem 6.3]{God}.) }
\label{lem:Hclassify}   \samepage
\ifDRAFT {\rm lem:Hclassify}. \fi
For a sequence 
\begin{equation}
(\{\th_i\}_{i=0}^d;
 \{ \th^*_i\}_{i=0}^d;
\{\phi_i\}_{i=1}^d)               \label{eq:parray2}
\end{equation}
of scalars in $\F$,
the following are equivalent:
\begin{itemize}
\item[\rm (i)]
there exists a Hessenberg system $\Phi$  on $V$ with parameter array \eqref{eq:parray2};
\item[\rm (ii)]
$\{\th_i\}_{i=0}^d$ are mutually distinct,
$\{\th^*_i\}_{i=0}^d$ are mutually distinct,
and $\{\phi_i\}_{i=1}^d$ are nonzero.
\end{itemize}
Assume that {\rm (i), (ii)} hold.
Then $\Phi$ is unique up to isomorphism of Hessenberg systems.
\end{lemma}

\section{Circular Hessenberg pairs and systems}
\label{sec:CHS}
\ifDRAFT {\rm sec:CHS}. \fi

In this section, we recall from \cite{JHL} the notion of a Circular Hessenberg pair and system.

\begin{defi}     {\rm (See \cite[Definition 2.9]{JHL}.) }
 \label{def:CHP}    \samepage
\ifDRAFT {\rm def:CHP}. \fi
By a {\em Circular Hessenberg pair} on $V$ we mean
an ordered pair $A,A^*$ of elements in $\text{End}(V)$ such that
\begin{itemize}
\item[\rm (i)]
there exists a basis for $V$ with respect to which
the matrix representing $A$ is CH and
the matrix representing $A^*$ is diagonal;
\item[\rm (ii)]
there exists a basis for $V$ with respect to which
the matrix representing $A^*$ is CH and
the matrix representing $A$ is diagonal.
\end{itemize}
\end{defi}

\begin{lemma}  {\rm (See \cite[Lemma 2.10]{JHL}.) }
\label{lem:CHPHP}    \samepage
\ifDRAFT {\rm lem:CHPHP}. \fi
Let $A,A^*$ denote a Circular Hessenberg  pair on $V$.
Then $A,A^*$ is a Hessenberg  pair on $V$.
Moreover, each of $A, A^*$ is multiplicity-free.
\end{lemma}

\begin{defi}     {\rm (See \cite[Definition 2.11]{JHL}.) }
\label{def:CHS}    \samepage
\ifDRAFT {\rm def:CHS}. \fi
By a {\em Circular Hessenberg system} (or {\em CHS}) on $V$
we mean a Hessenberg system
\[
 (A; \{E_i\}_{i=0}^d; A^*; \{E^*_i\}_{i=0}^d)
\]
on $V$ such that
\begin{itemize}
\item[\rm (i)]
$\displaystyle
E_i A^* E_j =
\begin{cases}
 0 & \text{  if  $1 < j-i <d$},
\\
\neq  0  &  \text{ if $ j-i=d$}
\end{cases}
\qquad\qquad (0 \leq i,j \leq d)$;
\item[\rm (ii)]
$\displaystyle
E^*_i A E^*_j =
\begin{cases}
 0 & \text{  if   $1 < j-i <d$},
\\
\neq  0  &  \text{ if  $ j-i=d$}
\end{cases}
\qquad\qquad (0 \leq i,j \leq d)$.
\end{itemize}
\end{defi}

Circular Hessenberg pairs and systems are related as follows.
Let  
\[
(A; \{E_i\}_{i=0}^d; A^*; \{E^*_i\}_{i=0}^d)
\]
denote a CHS  on $V$.
Then $A,A^*$ is a Circular Hessenberg  pair on $V$.
Conversely,
let $A,A^*$ denote a Circular Hessenberg pair on $V$.
Let $\{v_i\}_{i=0}^d$ (resp.\ $\{v^*_i\}_{i=0}^d$) denote a basis for $V$
that satisfies Definition \ref{def:CHP}(ii) (resp.\ Definition \ref{def:CHP}(i)).
For $0 \leq i \leq d$,
the vector $v_i$ (resp.\ $v^*_i$) is an eigenvector for $A$ (resp.\ $A^*$);
let $E_i$  (resp.\ $E^*_i$) denote the corresponding primitive idempotent of $A$ (resp.\ $A^*$).
Then   $(A; \{E_i\}_{i=0}^d; A^*; \{E^*_i\}_{i=0}^d)$
is a CHS on $V$.

We comment on the classification of CHS.
Under some additional conditions (see \cite[Definition 2.18]{JHL}),
the CHS are classified in \cite[Theorem 5.6]{JHL};
there are precisely four families, see   \cite[ Examples 5.1--5.4]{JHL}.
By \cite{NT:CHS}  the additional conditions are automatically satisfied.
This gives a classification of CHS.

\section{Quasi-Circular Hessenberg pairs and systems}
\label{sec:QCHS}
\ifDRAFT {\rm sec:QCHS}. \fi

In this section, we introduce the notion of a Quasi-Circular Hessenberg pair and system.

\begin{defi} 
 \label{def:QCHP}    \samepage
\ifDRAFT {\rm def:QCHP}. \fi
By a {\em Quasi-Circular Hessenberg  pair} on $V$ we mean
an ordered pair $A,A^*$ of elements  in $\text{End}(V)$ such that
\begin{itemize}
\item[\rm (i)]
there exists a basis for $V$ with respect to which
the matrix representing $A$ is QCH and
the matrix representing $A^*$ is diagonal;
\item[\rm (ii)]
there exists a basis for $V$ with respect to which
the matrix representing $A^*$ is QCH and
the matrix representing $A$ is diagonal.
\end{itemize}
\end{defi}

\begin{lemma}   \label{lem:QCHP}   \samepage
\ifDRAFT {\rm lem:QCHP}. \fi
Let $A,A^*$ denote a Quasi-Circular Hessenberg  pair on $V$.
Then $A,A^*$ is a Hessenberg  pair on $V$.
Moreover, each of $A,A^*$ is multiplicity-free.
\end{lemma}

\begin{proof}
The first assertion holds since a QCH matrix is a Hessenberg matrix.
The second assertion follows from  Lemma \ref{lem:mfree}.
\end{proof}

\begin{defi}  
\label{def:QCHS}    \samepage
\ifDRAFT {\rm def:QCHS}. \fi
By a {\em Quasi-Circular Hessenberg   system (or {\rm QCHS}}) on $V$
we mean a Hessenberg system
\[
 (A; \{E_i\}_{i=0}^d; A^*; \{E^*_i\}_{i=0}^d)
\]
on $V$ such that
\begin{itemize}
\item[\rm (i)]
$E_i A^* E_j =  0$ \quad  if  $1 < j-i <d \qquad (0 \leq i,j \leq d)$;
\item[\rm (ii)]
$E^*_i A E^*_j =  0$  \quad if  $1 < j-i <d \qquad (0 \leq i,j \leq d)$.
\end{itemize}
\end{defi}

Quasi-Circular Hessenberg  pairs and systems  are related as follows.
Let  
\[
(A; \{E_i\}_{i=0}^d; A^*; \{E^*_i\}_{i=0}^d)
\]
denote a QCHS on $V$.
Then $A,A^*$ is a Quasi-Circular Hessenberg  pair on $V$.
Conversely,
let $A,A^*$ denote a Quasi-Circular Hessenberg  pair on $V$.
Let $\{v_i\}_{i=0}^d$ (resp.\ $\{v^*_i\}_{i=0}^d$) denote a basis for $V$
that satisfies Definition \ref{def:QCHP}(ii) (resp.\ Definition \ref{def:QCHP}(i)).
For $0 \leq i \leq d$,
the vector $v_i$ (resp.\ $v^*_i$) is an eigenvector for $A$ (resp.\ $A^*$);
let $E_i$  (resp.\ $E^*_i$) denote the corresponding primitive idempotent of $A$ (resp.\ $A^*$).
Then   $(A; \{E_i\}_{i=0}^d; A^*; \{E^*_i\}_{i=0}^d)$
is a QCHS on $V$.

We comment on the classification of  QCHS.
In Section \ref{sec:TD} we will introduce a family of Hessenberg systems,
said  to be TD.
As we will describe in \eqref{eq:v}, 
the TD family and the QCHS family coincide.
The TD family is classified in Propositions \ref{prop:vth} and \ref{prop:TDd2}.
This gives  a classification of the QCHS.

\section{Leonard pairs and systems}
\label{sec:LS}
\ifDRAFT {\rm sec:LS}. \fi

In this section, we recall from \cite{T:Leonard} the notion of a Leonard pair
and Leonard system.

\begin{defi}    {\rm (See \cite[Definition 1.1]{T:Leonard}.) }
 \label{def:LP}    \samepage
\ifDRAFT {\rm def:LP}. \fi
By a {\em Leonard  pair} on $V$ we mean
an ordered pair $A,A^*$ of elements in $\text{End}(V)$ such that
\begin{itemize}
\item[\rm (i)]
there exists a basis for $V$ with respect to which
the matrix representing $A$ is Irreducible Tridiagonal and
the matrix representing $A^*$ is diagonal;
\item[\rm (ii)]
there exists a basis for $V$ with respect to which
the matrix representing $A^*$ is Irreducible Tridiagonal and
the matrix representing $A$ is diagonal.
\end{itemize}
\end{defi}

\begin{lemma}   \label{lem:LH}   \samepage
\ifDRAFT {\rm lem:LH}. \fi
Let $A,A^*$ denote a Leonard  pair on $V$.
Then $A,A^*$ is a Hessenberg  pair on $V$.
Moreover, each of $A,A^*$ is multiplicity-free.
\end{lemma}

\begin{proof}
The first assertion holds since an Irreducible Tridiagonal   matrix is a Hessenberg matrix.
The second assertion follows from  Lemma \ref{lem:mfree}.
\end{proof}

\begin{defi}   {\rm (See \cite[Definition 1.4]{T:Leonard}.) }
\label{def:LS}    \samepage
\ifDRAFT {\rm def:LS}. \fi
By a {\em Leonard  system} (or {\em LS}) on $V$
we mean a Hessenberg system
\[
 (A; \{E_i\}_{i=0}^d; A^*; \{E^*_i\}_{i=0}^d)
\]
on $V$ such that
\begin{itemize}
\item[\rm (i)] 
$\displaystyle
E_i A^* E_j = 
\begin{cases}
0 & \text{ if $1 < j-i$},
\\
\neq 0 & \text{ if $1 =j-i$}
\end{cases}
\qquad\qquad (0 \leq i,j \leq d)$;
\item[\rm (ii)] 
$\displaystyle
E^*_i A E^*_j = 
\begin{cases}
0 & \text{ if $1 < j-i$},
\\
\neq 0 & \text{ if $1=j-i$}
\end{cases}
\qquad\qquad (0 \leq i,j \leq d)$.
\end{itemize}
\end{defi}

Leonard  pairs and Leonard systems are related as follows.
Let  
\[
(A; \{E_i\}_{i=0}^d; A^*; \{E^*_i\}_{i=0}^d)
\]
denote a Leonard  system on $V$.
Then $A,A^*$ is a Leonard  pair on $V$.
Conversely,
let $A,A^*$ denote a Leonard  pair on $V$.
Let $\{v_i\}_{i=0}^d$ (resp.\ $\{v^*_i\}_{i=0}^d$) denote a basis for $V$
that satisfies Definition \ref{def:LP}(ii) (resp.\ Definition \ref{def:LP}(i)).
For $0 \leq i \leq d$,
the vector $v_i$ (resp.\ $v^*_i$) is an eigenvector for $A$ (resp.\ $A^*$);
let $E_i$  (resp.\ $E^*_i$) denote the corresponding primitive idempotent of $A$ (resp.\ $A^*$).
Then   $(A; \{E_i\}_{i=0}^d; A^*; \{E^*_i\}_{i=0}^d)$
is a Leonard system on $V$.

In the next result,
we  classify  the Leonard systems up to isomorphism.

\begin{lemma}   {\rm (See \cite[Theorem 1.9]{T:Leonard}.) }
\label{lem:LSclassify}   \samepage
\ifDRAFT {\rm lem:LSclassify}. \fi
Consider a sequence 
\begin{equation}
(\{\th_i\}_{i=0}^d;
 \{ \th^*_i\}_{i=0}^d;
\{\phi_i\}_{i=1}^d)               \label{eq:parray3}
\end{equation}
of scalars in $\F$.
Then there exists a Leonard system $\Phi$  on $V$ with parameter array \eqref{eq:parray3}
if and only if 
 the following {\rm (i)--(v)} hold:
\begin{itemize}
\item[\rm (i)]
$\{\th_i\}_{i=0}^d$ are mutually distinct,  and
$\{\th^*_i\}_{i=0}^d$ are mutually distinct;
\item[\rm (ii)]
$\phi_i \neq 0$ $(1 \leq i \leq d)$;
\item[\rm (iii)]
the scalars
\[
\frac{\th_{i-2} - \th_{i+1} } { \th_{i-1} - \th_i},
\qquad
\frac{\th^*_{i-2} - \th^*_{i+1} } { \th^*_{i-1} - \th^*_i}
\]
are equal and independent of $i$ for $2 \leq i \leq d-1$;
\item[\rm (iv)]
for $1 \leq i \leq d$ we have $\vphi_i \neq 0$, where
\[
  \vphi_i = \phi_1 \sum_{\ell=0}^{i-1} \frac{\th_\ell - \th_{d-\ell} } { \th_0 - \th_d} 
                  + (\th^*_i - \th^*_0)(\th_{i-1} - \th_d);
\]
\item[\rm (v)]
$\displaystyle
  \phi_i = \vphi_1 \sum_{\ell=0}^{i-1} \frac{\th_\ell - \th_{d-\ell} } { \th_{0} - \th_d} 
                  + (\th^*_i - \th^*_0)(\th_{d-i+1} - \th_0)  \qquad\qquad (1 \leq i \leq d)$.
\end{itemize}
Assume that {\rm (i)--(v)} hold.
Then $\Phi$ is unique up to isomorphism of Hessenberg systems.
\end{lemma}

\section{Tridiagonal Hessenberg pairs and systems}
\label{sec:THS}
\ifDRAFT {\rm sec:THS}. \fi

In this section, we introduce the notion of a Tridiagonal Hessenberg pair and system.
\begin{defi} 
 \label{def:THP}    \samepage
\ifDRAFT {\rm def:THP}. \fi
By a {\em Tridiagonal Hessenberg  pair} on $V$ we mean
an ordered pair $A,A^*$ of elements  in $\text{End}(V)$ such that
\begin{itemize}
\item[\rm (i)]
there exists a basis for $V$ with respect to which
the matrix representing $A$ is TH  and
the matrix representing $A^*$ is diagonal;
\item[\rm (ii)]
there exists a basis for $V$ with respect to which
the matrix representing $A^*$ is TH and
the matrix representing $A$ is diagonal.
\end{itemize}
\end{defi}

\begin{lemma}   \label{lem:QLQCH}   \samepage
\ifDRAFT {\rm lem:QLQCH}. \fi
Let $A,A^*$ denote a Tridiagonal Hessenberg   pair on $V$.
Then $A,A^*$ is a Hessenberg pair on $V$.
Moreover, each of $A,A^*$ is multiplicity-free.
\end{lemma}

\begin{proof}
The first assertion holds since a TH matrix  is a Hessenberg matrix.
The second assertion follows from  Lemma \ref{lem:mfree}.
\end{proof}

\begin{defi}  
\label{def:THS}    \samepage
\ifDRAFT {\rm def:THS}. \fi
By a {\em Tridiagonal Hessenberg   system} (or {\em THS}) on $V$
we mean a Hessenberg system
\[
 (A; \{E_i\}_{i=0}^d; A^*; \{E^*_i\}_{i=0}^d)
\]
such that
\begin{itemize}
\item[\rm (i)] 
$E_i A^* E_j = 0$ \quad if $1 < j-i$ \quad $(0 \leq i,j \leq d)$;
\item[\rm (ii)] 
$E^*_i A E^*_j = 0$ \quad if $1 < j-i$ \quad $(0 \leq i,j \leq d)$.
\end{itemize}
\end{defi}

Tridiagonal Hessenberg  pairs and systems are related as follows.
Let  
\[
(A; \{E_i\}_{i=0}^d; A^*; \{E^*_i\}_{i=0}^d)
\]
denote a THS on $V$.
Then $A,A^*$ is a Tridiagonal Hessenberg pair on $V$.
Conversely,
let $A,A^*$ denote a Tridiagonal Hessenberg  pair on $V$.
Let $\{v_i\}_{i=0}^d$ (resp.\ $\{v^*_i\}_{i=0}^d$) denote a basis for $V$
that satisfies Definition \ref{def:THP}(ii) (resp.\ Definition \ref{def:THP}(i)).
For $0 \leq i \leq d$,
the vector $v_i$ (resp.\ $v^*_i$) is an eigenvector for $A$ (resp.\ $A^*$);
let $E_i$  (resp.\ $E^*_i$) denote the corresponding primitive idempotent of $A$ (resp.\ $A^*$).
Then   $(A; \{E_i\}_{i=0}^d; A^*; \{E^*_i\}_{i=0}^d)$
is a THS on $V$.

\bigskip

We have introduced some families of Hessenberg systems.
We comment on how these families are related.
The following diagram shows the logical implications:
\begin{equation}                   \label{eq:diag1b}
\text{
\begin{tabular}{cccccccc}
& &   &  & CHS
\\
& & & & $\Downarrow$
\\
LS
&
$\Rightarrow$ 
&
THS
&
$\Rightarrow$
& 
QCHS
&
$\Rightarrow$ 
&
HS
\end{tabular}
}
\end{equation}

\section{The TD property and REC property}
\label{sec:TD}
\ifDRAFT {\rm sec:TD}. \fi

We now turn our attention to a pair of relations called the tridiagonal relations.
These relations are introduced in the following result.

\begin{lemma}   {\rm (See \cite[Theorem 1.12]{T:Leonard}.) }
\label{lem:LPTD}    \samepage
\ifDRAFT {\rm lem:LPTD}. \fi
Let $A,A^*$ denote a Leonard pair on $V$.
Then there exist scalars 
$\beta, \gamma, \gamma^*, \varrho, \varrho^*$ in $\F$
such that both
\begin{align}
0 &=[A, A^2 A^* - \beta A A^* A + A^* A^2 - \gamma (A A^* + A^* A) - \varrho A^*],            \label{eq:td1}
\\
0&= [A^*, A^{*2} A - \beta A^* A A^* + A A^{*2} - \gamma^* (A^* A + A A^*) - \varrho^* A].     \label{eq:td2}
\end{align}
Moreover the sequence
$\beta, \gamma, \gamma^*, \varrho, \varrho^*$
is uniquely determined by $A,A^*$ provided that $d \geq 3$.
\end{lemma}

The relations \eqref{eq:td1}, \eqref{eq:td2} are called the {\em tridiagonal relations}
 (or {\em TD relations}).

\begin{defi}    \label{def:TDH}    \samepage
\ifDRAFT {\rm def:TDH}. \fi
Let
$\Phi = (A; \{E_i\}_{i=0}^d; A^*; \{E^*_i\}_{i=0}^d)$ denote a Hessenberg system  on $V$.
Then $\Phi$ is said to be {\em TD}
whenever
there exist scalars $\beta, \gamma, \gamma^*, \varrho, \varrho^*$ in $\F$
such that 
$A,A^*$ satisfy \eqref{eq:td1}, \eqref{eq:td2}. 
\end{defi}

By Lemma \ref{lem:LPTD} we have the  implication \; LS $\Rightarrow$ TD.

\begin{defi}    \label{def:recH}    \samepage
\ifDRAFT {\rm def:recH}. \fi
Let $\Phi = (A; \{E_i\}_{i=0}^d; A^*; \{E^*_i\}_{i=0}^d)$ denote a Hessenberg system on $V$,
with eigenvalue sequence $\{\th_i\}_{i=0}^d$
and
dual eigenvalue sequence $\{\th^*_i\}_{i=0}^d$.
For $\beta \in \F$,
$\Phi$ is said to be {\em $\beta$-recurrent}
whenever each of  $\{\th_i\}_{i=0}^d$ and  $\{\th^*_i\}_{i=0}^d$
is $\beta$-recurrent.
We say that $\Phi$ is {\em recurrent} (or {\em REC}) whenever
there exists  $\beta \in \F$ such that
 $\Phi$ is $\beta$-recurrent.
\end{defi}

In \cite{NT:CHS} we obtained the implication
\begin{equation}
\text{CHS  $\Rightarrow$ TD}.  \label{eq:i}
\end{equation}
We will prove the following:
\begin{align}
& \text{TD $\Rightarrow$ REC},  \label{eq:ii}
\\
&\text{THS $\Rightarrow$ TD},  \label{eq:iii}
\\
\text{QCHS is a } & \text{disjoint union of CHS and THS},    \label{eq:iv}
\\
&\text{QCHS  $\Leftrightarrow$ TD}.   \label{eq:v}
\end{align}
The proofs of \eqref{eq:ii}--\eqref{eq:v} are given in Section \ref{sec:proof}.

\bigskip
By \eqref{eq:diag1b} and \eqref{eq:ii}--\eqref{eq:v},
 we get the following implications:
\[
\text{QCHS} \quad   \Leftrightarrow  \quad 
\text{TD} 
\]
\[
\text{LS} \quad \Rightarrow \quad 
\text{THS} \quad  \Rightarrow \quad 
\text{TD} \quad \Rightarrow \quad  
\text{REC} \quad  \Rightarrow  \quad   
\text{HS}   
\]
By \eqref{eq:iv} and \eqref{eq:v},
CHS is the complement of THS in TD.

\section{Adjusting the split sequence}
\label{sec:adjust}
\ifDRAFT {\rm sec:adjust}. \fi

In this section, we adjust the split sequence using the approach
of \cite{JHL}.
Let
\[
\Phi = ( A; \{E_i\}_{i=0}^d;A^*; \{E^*_i\}_{i=0}^d)
\]
denote a Hessenberg system on $V$, with parameter array
$(\{\th_i\}_{i=0}^d; \{\th^*_i\}_{i=0}^d; \{\phi_i\}_{i=1}^d)$.

\begin{defi}     {\rm (See \cite[line (6)]{JHL}.)}
\label{def:vth}    \samepage
\ifDRAFT {\rm def:vth}. \fi
Define scalars
\[
\vartheta_i = \phi_i - (\th^*_i - \th^*_0) ( \th_{d-i+1} - \th_0)
\qquad\qquad (1 \leq i \leq d)
\]
and
\[
\vth_0 = 0,
\qquad\qquad
\vth_{d+1} = 0.
\]
\end{defi}

\begin{lemma}   \label{lem:vths}    \samepage
\ifDRAFT {\rm lem:vths}. \fi
We have
$\vartheta^*_i = \vartheta_{d-i+1}
\quad
(0 \leq  i \leq d+1)$.
\end{lemma}

\begin{proof}
By Lemma \ref{lem:Phis}.
\end{proof}

\begin{lemma}   \label{lem:vths2}    \samepage
\ifDRAFT {\rm lem:vths2}. \fi
The following are equivalent:
\begin{itemize}
\item[\rm (i)]
$\{\vartheta_i\}_{i=0}^{d+1}$ is $\beta$-recurrent;
\item[\rm (ii)]
$\{\vartheta^*_i\}_{i=0}^{d+1}$ is $\beta$-recurrent.
\end{itemize}
\end{lemma}

\begin{proof}
Use Lemma  \ref{lem:vths}.
\end{proof}

\begin{lemma}  \label{lem:notevthd2}    \samepage
\ifDRAFT {\rm lem:notevthd2}. \fi
Assume that $d=2$.
Then  for $\beta \in \F$ the following are equivalent:
\begin{itemize}
\item[\rm (i)]
the sequence $\{\vth_i\}_{i=0}^{d+1}$ is $\beta$-recurrent;
\item[\rm (ii)]
$\beta = -1$ or $\vth_1 = \vth_2$.
\end{itemize}
\end{lemma}

\begin{proof}
The sequence $\{\vth_i\}_{i=0}^{d+1}$ is $\beta$-recurrent if and only if
\[
\vth_0 - (\beta+1) \vth_1 + (\beta+1) \vth_2 - \vth_3 = 0.
\]
By Definition \ref{def:vth} we have $\vth_0 = 0$ and $\vth_3 = 0$.
The result follows.
\end{proof}

The following result is a variation on  \cite[Lemma 12.5]{T:Leonard}.

\begin{lemma} \label{lem:TL}     \samepage
\ifDRAFT {\rm lem:TL}. \fi
Pick scalars  $\beta, \gamma, \varrho \in \F$.
Then $A,A^*$ satisfy \eqref{eq:td1}
if and only if the following  {\rm (i)--(iii)} hold:
\begin{itemize}
\item[\rm (i)]
$\{\th_i\}_{i=0}^d$ is $(\beta, \gamma, \varrho)$-recurrent;
\item[\rm (ii)]
$\{\th^*_i\}_{i=0}^d$ is $\beta$-recurrent;
\item[\rm (iii)]
$\{\vartheta_i\}_{i=0}^{d+1}$ is $\beta$-recurrent.
\end{itemize}
\end{lemma}

\begin{proof}
We may assume that $A,A^*$ are the matrices from  Lemma \ref{lem:matrixAAs}(iv).
Now use \cite[Lemma 12.5]{T:Leonard}.
\end{proof}

\begin{lemma} \label{lem:TL2}     \samepage
\ifDRAFT {\rm lem:TL2}. \fi
Pick scalars  $\beta, \gamma^*, \varrho^* \in \F$.
Then $A,A^*$ satisfy \eqref{eq:td2}
if and only if the following  {\rm (i)--(iii)} hold:
\begin{itemize}
\item[\rm (i)]
$\{\th^*_i\}_{i=0}^d$ is $(\beta, \gamma^*, \varrho^*)$-recurrent;
\item[\rm (ii)]
$\{\th_i\}_{i=0}^d$ is $\beta$-recurrent;
\item[\rm (iii)]
$\{\vartheta_i\}_{i=0}^{d+1}$ is $\beta$-recurrent.
\end{itemize}
\end{lemma}

\begin{proof}
Apply Lemma \ref{lem:TL} to $\Phi^*$ using Lemma \ref{lem:vths2}.
\end{proof}

\begin{prop}  {\em (See \cite[Lemma 2.16]{JHL}.) }
\label{prop:vth}  \samepage
\ifDRAFT {\rm prop:vth}. \fi
Assume that $d \geq 3$.
Then the following are equivalent:
\begin{itemize}
\item[\rm (i)]
$\Phi$ is TD;
\item[\rm (ii)]
there exists $\beta \in \F$ such that each of 
$\{\th_i\}_{i=0}^d$,
$\{\th^*_i\}_{i=0}^d$,
$\{\vartheta_i\}_{i=0}^{d+1}$ is $\beta$-recurrent.
\end{itemize}
\end{prop}

\begin{prop}   \label{prop:TDd2}   \samepage
\ifDRAFT {\rm prop:TDd2}. \fi
Assume that $d=2$. 
Then $\Phi$ is TD.
\end{prop}

\begin{proof}
Define  $\beta = -1$ and
\begin{align*}
\gamma &= \th_0 - \beta \th_1 + \th_2,
&
\gamma^* &= \th^*_0 - \beta \th^*_1 + \th^*_2,
\\
\varrho &= \th_0^2 - \beta \th_0 \th_1 + \th_1^2 - \gamma (\th_0 + \th_1),
&
\varrho^* &= \th_0^{*2} - \beta \th^*_0 \th^*_1 + \th_1^{*2} - \gamma^* (\th^*_0 + \th^*_1).
\end{align*}
The sequence $\{\th_i\}_{i=0}^d$ is $(\beta,\gamma)$-recurrent by construction,
and $(\beta, \gamma, \varrho)$-recurrent by Lemma \ref{lem:recseq2}(i).
The sequence $\{\th^*_i\}_{i=0}^d$ is $(\beta,\gamma^*)$-recurrent by construction,
and $(\beta, \gamma^*, \varrho^*)$-recurrent by Lemma \ref{lem:recseq2}(i).
The sequence $\{\vth_i\}_{i=0}^{d+1}$ is $\beta$-recurrent by Lemma \ref{lem:notevthd2}.
The maps $A,A^*$ satisfy \eqref{eq:td1} by Lemma \ref{lem:TL} and 
\eqref{eq:td2} by Lemma \ref{lem:TL2}.
Therefore  $\Phi$ is TD by Definition \ref{def:TDH}.
\end{proof}

\section{Vanishing triple products}
\label{sec:triple}
\ifDRAFT {\rm sec:triple}. \fi

Throughout this section,
let $\Phi = ( A; \{E_i\}_{i=0}^d;$ $ A^*; \{E^*_i\}_{i=0}^d)$
denote a Hessenberg system on $V$ with parameter array
$(\{\th_i\}_{i=0}^d; \{\th^*_i\}_{i=0}^d$; $\{\phi_i\}_{i=1}^d)$.
We consider the triple  products
\[
E_i A^* E_j,
\qquad\qquad
E^*_i A E^*_j
\qquad\qquad
(0 \leq i,j \leq d, \; 1 < j-i).
\]
We obtain some results concerning which of these triple products is zero.
We use these results to describe TD and THS.

\begin{lemma}  {\rm (See \cite[Lemma 4.4]{JHL}.)} 
\label{lem:JHL}     \samepage
\ifDRAFT {\rm lem:JHL}. \fi
Assume that $d \geq 3$ and there exists $\beta \in \F$ such that each of
$\{\th_i\}_{i=0}^d$, $\{\th^*_i\}_{i=0}^d$, $\{\vartheta_i\}_{i=0}^{d+1}$
is $\beta$-recurrent.
Then the following {\rm (i)--(iv)} hold:
\begin{itemize}
\item[\rm (i)]
$E_i A^* E_j = 0$ \quad if \quad $1 < j-i < d$ \quad $(0 \leq i,j \leq d)$;
\item[\rm (ii)]
$E^*_i A E^*_j = 0$ \quad if \quad $1 < j-i < d$ \quad $(0 \leq i,j \leq d)$;
\item[\rm (iii)]
$E_0 A^* E_d = 0$ \quad if and only if \quad $\vartheta_1 = \vartheta_d$;
\item[\rm (iv)]
$E^*_0 A E^*_d = 0$ \quad if and only if \quad $\vartheta_1 = \vartheta_d$.
\end{itemize}
\end{lemma}

\begin{prop} \label{prop:QCHTD}    \samepage 
\ifDRAFT {\rm prop:QCHTD}. \fi
For $d \geq 3$ the following are equivalent:
\begin{itemize}
\item[\rm (i)]
$E_i A^* E_j = 0$\quad if \quad $1 < j-i$ \quad $(0 \leq i,j \leq d)$;
\item[\rm (ii)]
$E^*_i A E^*_j = 0$ \quad if \quad $1 < j-i$ \quad $(0 \leq i,j \leq d)$;
\item[\rm (iii)]
$\vth_1 = \vth_d$
and 
there exists $\beta \in \F$ such that
each of 
$\{\th_i\}_{i=0}^d$,
$\{\th^*_i\}_{i=0}^d$,
$\{\vth_i\}_{i=0}^{d+1}$
is $\beta$-recurrent;
\item[\rm (iv)]
$\Phi$  is THS.
\end{itemize}
\end{prop}

\begin{proof}
(i) $\Rightarrow$ (iii)
We may assume that $A,A^*$ are the matrices from  Lemma \ref{lem:matrixAAs}(iv).
By \cite[Lemma 12.2]{T:Leonard} there exist $\beta, \gamma, \varrho \in \F$
such that $A,A^*$ satisfy \eqref{eq:td1}.
By this and Lemma \ref{lem:TL}, 
$\{\th_i\}_{i=0}^d$ is $(\beta,\gamma, \varrho)$-recurrent,
and each of 
$\{\th^*_i\}_{i=0}^d$,
$\{\vth_i\}_{i=0}^{d+1}$ 
is $\beta$-recurrent.
By Lemmas \ref{lem:recseq1}, \ref{lem:recseq2}
the sequence $\{\th_i\}_{i=0}^d$ is $\beta$-recurrent.
By Lemma \ref{lem:JHL}(iii), $\vth_1 = \vth_d$.

(iii) $\Rightarrow$ (i)
By lemma \ref{lem:JHL}.

(ii) $\Leftrightarrow$ (iii)
Apply  (i) $\Leftrightarrow$ (iii) to $\Phi^*$.

(i),(ii) $\Leftrightarrow$ (iv)
By Definition \ref{def:THS}.
\end{proof}

We bring in some notation.
Let $\lambda$ denote an indeterminate.
Let the $\F$-algebra $\F[\lambda]$ consist of the polynomials in $\lambda$
that have all coefficients in $\F$.
For $0 \leq i \leq d$ define $\tau_i$, $\tau^*_i$, $\eta_i$, $\eta^*_i \in \F[\lambda]$
by
\begin{align*}
\tau_i &= (\lambda - \th_0)(\lambda -\th_1) \cdots (\lambda-\th_{i-1}),
&
\tau^*_i &= (\lambda - \th^*_0)(\lambda-\th^*_1) \cdots (\lambda-\th^*_{i-1}),
\\
\eta_i &= (\lambda - \th_d)(\lambda-\th_{d-1}) \cdots (\lambda-\th_{d-i+1}),
&
\eta^*_i &= (\lambda- \th^*_d)(\lambda-\th^*_{d-1}) \cdots (\lambda-\th^*_{d-i+1}).
\end{align*}

\begin{defi}   \label{def:flat}    \samepage
\ifDRAFT {\rm def:flat}. \fi
For $X \in \text{\rm End}(V)$ let $X^\flat$ denote
the matrix in $\Mat_{d+1}(\F)$ that represents $X$ with respect to an inverted $\Phi^*$-split basis
for $V$.
\end{defi}

The matrices $A^\flat$, $A^{* \flat}$ are given in Lemma \ref{lem:matrixAAs}(iv).
The following result is a variation on \cite[Lemmas 4.4, 4.6]{T:Leonard}.

\begin{lemma} \label{lem:matrixEEs}    \samepage
\ifDRAFT {\rm lem:matrixEEs}. \fi
For $0 \leq r \leq d$, the $(i,j)$-entry of $E_r^\flat$ and $E_r^{*\flat}$ are
\begin{align*}
(E_r^\flat)_{i,j} &=
\frac{ \eta_j (\th_{r}) \tau_{d-i}(\th_{r}) }
       { \eta_{d-r} (\th_{r}) \tau_{r} (\th_{r}) }
\qquad\qquad (0 \leq i,j \leq d),
\\
(E_r^{* \flat})_{i,j} &=
 \frac{ \phi_1 \phi_2 \cdots \phi_j }
        { \phi_1 \phi_2 \cdots \phi_i}
 \frac{\tau^*_i (\th^*_r) \eta^*_{d-j} (\th^*_r) }
        {\tau^*_r (\th^*_r) \eta^*_{d-r} (\th^*_r) }
\qquad\qquad   (0 \leq i,j \leq d).
\end{align*}
\end{lemma}

\begin{proof}
We may assume that $A,A^*$ are the matrices from Lemma \ref{lem:matrixAAs}(iv).
Now use \cite[Lemmas 4.4, 4.6]{T:Leonard}.
\end{proof}

Until further notice, we  fix an integer $r$ $(2 \leq r \leq d)$ and 
consider the entries of
\begin{equation}
(E_{d-r} A^* E_{d-r+2})^\flat.            \label{eq:triple}
\end{equation}

\begin{lemma}   \label{lem:entrypre}   \samepage
\ifDRAFT {\rm lem:entrypre}. \fi
For $0 \leq i,j \leq d$ the $(i,j)$-entry of \eqref{eq:triple} is 
\[
 \frac{\tau_{d-i} (\th_{d-r}) \eta_j (\th_{d-r+2}) }
        {\tau_{d-r} (\th_{d-r}) \eta_{r-2} (\th_{d-r+2}) }
\]
times
the $(r,r-2)$-entry of \eqref{eq:triple}.
\end{lemma}

\begin{proof}
Routine verification using Lemmas \ref{lem:matrixAAs}(iv) and \ref{lem:matrixEEs}.
\end{proof}

\begin{lemma}  \label{lem:entry0d2}   \samepage
\ifDRAFT {\rm lem:entry0d2}. \fi
Assume that $d=2$.
Then the $(2,0)$-entry of $(E_0 A^* E_2)^\flat$ is
\begin{align*}
 & - \frac{\phi_1 - \phi_2} 
             {(\th_0 - \th_1)(\th_0 - \th_2)(\th_1 - \th_2)}
\\ &\qquad\qquad
+ \frac{\th^*_0 } { (\th_0 - \th_1)(\th_0 - \th_2) }
- \frac{\th^*_1 } { (\th_0 - \th_1)(\th_1 - \th_2) }
+ \frac{\th^*_2 } { (\th_0 - \th_2)(\th_1 - \th_2) }.
\end{align*}
All other entries are $0$.
\end{lemma}

\begin{proof}
Routine using Lemmas \ref{lem:matrixAAs}(iv) and \ref{lem:matrixEEs}.
\end{proof}

The next result is a variation on Lemma \ref{lem:entry0d2}.

\begin{lemma}   \label{lem:entry0dno2d2} \samepage
\ifDRAFT {lem:entry0dno2d2}. \fi
Assume that $d=2$.
Then the $(2,0)$-entry of $(E_0 A^* E_2)^\flat$ is
\[
- \frac{\vth_1 - \vth_2} {(\th_0 - \th_1)(\th_0 - \th_2)(\th_1 - \th_2)}.
\]
All other entries are $0$.
\end{lemma}

\begin{proof}
Routine verification using 
Definition \ref{def:vth} and Lemma \ref{lem:entry0d2}.
\end{proof}

\begin{prop}    \label{prop:vthd2}   \samepage
\ifDRAFT {\rm prop:vthd2}. \fi
Assume that $d=2$.
Then the following are equivalent:
\begin{itemize}
\item[\rm (i)]
$E_0 A^* E_2 = 0$;
\item[\rm (ii)]
$E^*_0 A E^*_2 = 0$;
\item[\rm (iii)]
$\vth_1 = \vth_2$;
\item[\rm (iv)]
$\Phi$ is THS.
\end{itemize}
\end{prop}

\begin{proof}
(i) $\Leftrightarrow$ (iii)
By Lemma \ref{lem:entry0dno2d2}.

(ii) $\Leftrightarrow$ (iii)
Apply (i) $\Leftrightarrow$ (iii) to $\Phi^*$.

(i),(ii) $\Leftrightarrow$ (iv)
By Definition \ref{def:THS}.
\end{proof}

\begin{lemma}   \label{lem:entry0}    \samepage
\ifDRAFT {\rm lem:entry0}. \fi
Assume that $d \geq 3$.
Then the $(r, r-2)$-entry of \eqref{eq:triple} is as follows.
For $r=2$  the entry is
\[
\frac{1} {
(\th_{d-3} - \th_d)(\th_{d-2} - \th_d)(\th_{d-1} - \th_d) }
\]
times
\begin{align*}
&
 - \frac{\th_{d-3}-\th_d }{\th_{d-2} - \th_{d-1} } \phi_1
 + \frac{\th_{d-3}-\th_d }{\th_{d-2} - \th_{d-1} } \phi_2
  -\phi_3
\\ & \qquad\qquad\qquad
 + (\th_{d-3} - \th_d)
 \bigg(
  \frac{\th_{d-1} - \th_d } { \th_{d-2} - \th_{d-1}} \th^*_0
 -  \frac{\th_{d-2} - \th_d } { \th_{d-2} - \th_{d-1}} \th^*_1
 + \th^*_2
 \bigg).
\end{align*}
For $3 \leq r \leq d-1$ the entry is
\[
\frac{1 } {
(\th_{d-r} - \th_{d-r+1})(\th_{d-r} - \th_{d-r+2})(\th_{d-r} - \th_{d-r+3}) 
         }
\]
times
\begin{align*}
& \phi_{r-2} 
  - \frac{\th_{d-r} - \th_{d-r+3} } { \th_{d-r+1} - \th_{d-r+2} } \phi_{r-1}
  + \frac{\th_{d-r} - \th_{d-r+3} } { \th_{d-r+1} - \th_{d-r+2} } \phi_{r}
  - \frac{ (\th_{d-r}- \th_{d-r+1}) (\th_{d-r} - \th_{d-r+3}) }
            { (\th_{d-r-1} - \th_{d-r+2})(\th_{d-r+1} - \th_{d-r+2}) } \phi_{r+1}
\\ & \qquad \qquad \qquad \qquad
 + (\th_{d-r}- \th_{d-r+3})
 \bigg(
   \th^*_{r-2}
 - \frac{\th_{d-r} - \th_{d-r+2}} { \th_{d-r+1} - \th_{d-r+2} } \th^*_{r-1}
 + \frac{\th_{d-r} - \th_{d-r+1} } { \th_{d-r+1} - \th_{d-r+2} } \th^*_{r} 
\bigg).
\end{align*}
For $r=d$ the entry is
\[
\frac{1} {
(\th_0 - \th_1)
(\th_0 - \th_2)
(\th_0 - \th_3)
 }
\]
times
\begin{align*}
&\phi_{d-2}
- \frac{\th_0 -\th _3} { \th_1 - \th_2 } \phi_{d-1}
+  \frac{\th_0 -\th _3} { \th_1 - \th_2 } \phi_{d}
+(\th_0 - \th_3)
\bigg(
\th^*_{d-2} 
- \frac{\th_0 - \th_2} { \th_1 - \th_2 } \th^*_{d-1}
+ \frac{\th_0 - \th_1} { \th_1 - \th_2 } \th^*_{d}
\bigg).
\end{align*}

\end{lemma}

\begin{proof}
Routine verification using Lemmas \ref{lem:matrixAAs}(iv) and \ref{lem:matrixEEs}.
\end{proof}

The next result is a variation on Lemma \ref{lem:entry0}.

\begin{lemma}   \label{lem:entry}    \samepage
\ifDRAFT {\rm lem:entry}. \fi
Assume that $d \geq 3$.
Then the $(r, r-2)$-entry of \eqref{eq:triple}
is as follows.
For $r=2$  the entry is
\[
\frac{1} {(\th_{d-3} - \th_d)(\th_{d-2} - \th_d)(\th_{d-1} - \th_d) }
\]
times
\begin{align*}
& (\th_0 - \th_{d-2} )
\bigg(
- \th^*_0 
+ \frac{\th_{d-3}- \th_d }{ \th_{d-2} - \th_{d-1} } \th^*_1
- \frac{\th_{d-3}- \th_d }{ \th_{d-2} - \th_{d-1} } \th^*_2
+\th^*_3
\bigg)
\\ & \qquad \qquad \qquad \qquad \qquad \qquad \qquad
- \frac{\th_{d-3}- \th_d }{ \th_{d-2} - \th_{d-1} }  \vth_1
+ \frac{\th_{d-3}- \th_d }{ \th_{d-2} - \th_{d-1} }  \vth_2
-\vth_3.
\end{align*}
For $3 \leq r \leq d-1$ the entry is
\[
\frac{1 } {
(\th_{d-r} - \th_{d-r+1})(\th_{d-r} - \th_{d-r+2})(\th_{d-r} - \th_{d-r+3}) 
      }
\]
times
\begin{align*}
& (\th_0 - \th_{d-r})\bigg(
  1 - \frac{(\th_{d-r} - \th_{d-r+1}) ( \th_{d-r} - \th_{d-r+3} ) }
               { (\th_{d-r-1} - \th_{d-r+2}) (\th_{d-r+1} - \th_{d-r+2}) } \bigg) \th^*_0 
\\ & \qquad
+ (\th_0 - \th_{d-r})\bigg(
 - \th^*_{r-2}
 + \frac {\th_{d-r} - \th_{d-r+3} }{ \th_{d-r+1} - \th_{d-r+2} } \th^*_{r-1}
 -  \frac {\th_{d-r} - \th_{d-r+3} }{ \th_{d-r+1} - \th_{d-r+2} } \th^*_{r}
\\ & \qquad \qquad \qquad \qquad \qquad \qquad \qquad \qquad
  + \frac{ (\th_{d-r} - \th_{d-r+1}) (\th_{d-r} - \th_{d-r+3}  ) }
           { (\th_{d-r-1} - \th_{d-r+2} ) (\th_{d-r+1} - \th_{d-r+2}) } \th^*_{r+1} \bigg)
\\ & \qquad
+ \vth_{r-2} 
- \frac{\th_{d-r} - \th_{d-r+3} } { \th_{d-r+1} - \th_{d-r+2} } \vth_{r-1} 
+ \frac{\th_{d-r} - \th_{d-r+3} } { \th_{d-r+1} - \th_{d-r+2} } \vth_{r} 
\\ & \qquad \qquad \qquad \qquad \qquad \qquad \qquad \qquad
- \frac{(\th_{d-r} - \th_{d-r+1} ) ( \th_{d-r} - \th_{d-r+3} ) }
          {(\th_{d-r-1} - \th_{d-r+2}) (\th_{d-r+1} - \th_{d-r+2}) } \vth_{r+1}.
\end{align*}
For $r=d$ the entry is
\[
\frac{1} {
(\th_0 - \th_1)
(\th_0 - \th_2)
(\th_0 - \th_3)
}
\]
times
\[
\vth_{d-2} 
- \frac{\th_0 - \th_3} { \th_1 - \th_2} \vth_{d-1}
+ \frac{\th_0 - \th_3} { \th_1 - \th_2} \vth_{d}.
\]
\end{lemma}

\begin{proof}
Routine verification using 
Definition \ref{def:vth} and Lemma \ref{lem:entry0}.
\end{proof}

\begin{lemma}   \label{lem:entryrec}    \samepage
\ifDRAFT {\rm lem:entryrec}. \fi
Assume that  $d \geq 3$ and
there exists $\beta \in \F$ such that
each of
$\{\th_i\}_{i=0}^d$,
$\{\th^*_i\}_{i=0}^d$
is $\beta$-recurrent.
Then the $(r,r-2)$-entry of \eqref{eq:triple} is as follows.
For $r=2$ the entry is
\[
\frac{- (\beta+1) \vth_1 + (\beta+1) \vth_2 - \vth_3}
        {(\th_{d-3} - \th_d)(\th_{d-2} - \th_d)(\th_{d-1} - \th_d) }.
\]
For $2 \leq r \leq d-1$ the entry is
\[
\frac{\vth_{r-2} - (\beta+1) \vth_{r-1} + (\beta+1) \vth_r - \vth_{r+1} } 
       {(\th_{d-r} - \th_{d-r+1})(\th_{d-r} - \th_{d-r+2})(\th_{d-r} - \th_{d-r+3})}.
\]
For $r=d$ the entry is
\[
\frac{\vth_{d-2} - (\beta+1) \vth_{d-1} + (\beta+1) \vth_d} {(\th_0 - \th_1)
(\th_0 - \th_2)
(\th_0 - \th_3)
}.
\]
\end{lemma}

\begin{proof}
By Lemma \ref{lem:entry}.
\end{proof}

\begin{prop}   \label{prop:vthrec}   \samepage
\ifDRAFT {\rm prop:vthrec}. \fi
Assume that $d \geq 3$ and
there exists $\beta \in \F$ such that
each of
$\{\th_i\}_{i=0}^d$,
$\{\th^*_i\}_{i=0}^d$
is $\beta$-recurrent.
Then the following are equivalent:
\begin{itemize}
\item[\rm (i)]
$E_{r-2} A^* E_r = 0$  $(2 \leq r \leq d)$;
\item[\rm (ii)]
$E^*_{r-2} A E^*_r = 0$  $(2 \leq r \leq d)$;
\item[\rm (iii)]
 $\{\vartheta_i\}_{i=0}^{d+1}$ is  $\beta$-recurrent.
\end{itemize}
Assume that {\rm (i)--(iii)} hold.
Then $\Phi$ is TD.
\end{prop}

\begin{proof}
(i) $\Leftrightarrow$ (iii)
Recall that $\vth_0 = 0$ and $\vth_{d+1}=0$.
Now use Lemmas \ref{lem:entrypre} and \ref{lem:entryrec}. 

(ii) $\Leftrightarrow$ (iii)
Apply (i) $\Leftrightarrow$ (iii) to $\Phi^*$.

Assume that {\rm (i)--(iii)} hold.
Then $\Phi$ is TD by Proposition \ref{prop:vth}.
\end{proof}

\bigskip

\begin{prop}   \label{prop:vthgpre}   \samepage
\ifDRAFT {\rm prop:vthgpre}. \fi
Assume that $d \geq 3$ and  there exists $\beta \in \F$ such that
each of
$\{\th_i\}_{i=0}^d$,
$\{\th^*_i\}_{i=0}^d$
is $\beta$-recurrent.
Then the following are equivalent:
\begin{itemize}
\item[\rm (i)]
$E_{r-2} A^* E_r = 0$ $(2 \leq r \leq d)$ and $E_0 A^* E_d = 0$;
\item[\rm (ii)]
$E^*_{r-2} A E^*_r = 0$ $(2 \leq r \leq d)$ and $E^*_0 A E^*_d = 0$;
\item[\rm (iii)]
$\{\vartheta_i\}_{i=0}^{d+1}$ is $\beta$-recurrent and $\vartheta_1 = \vartheta_d$.
\end{itemize}
Assume that {\rm (i)--(iii)} hold.
Then $\Phi$ is THS.
\end{prop}

\begin{proof}
(i) $\Leftrightarrow$ (iii)
By Lemma \ref{lem:JHL} and Proposition \ref{prop:vthrec}.

(ii) $\Leftrightarrow$ (iii)
Apply (i) $\Leftrightarrow$ (iii) to $\Phi^*$.

Assume that {\rm (i)--(iii)} hold.
Then $\Phi$ is THS by Proposition \ref{prop:QCHTD}.
\end{proof}

\section{Proof of \eqref{eq:ii}--\eqref{eq:v}}
\label{sec:proof}
\ifDRAFT {\rm sec:proof}. \fi

In this section we prove  \eqref{eq:ii}--\eqref{eq:v}.
Let
$\Phi = (A; \{E_i\}_{i=0}^d; A^*; \{E^*_i\}_{i=0}^d)$
denote a Hessenberg system on $V$,  with parameter array
$(\{\th_i\}_{i=0}^d; \{\th^*_i\}_{i=0}^d; \{\phi_i\}_{i=1}^d)$.
Let the scalars $\{\vartheta_i\}_{i=0}^{d+1}$ be from Definition \ref{def:vth}.

\begin{proofof}{\eqref{eq:ii}}
Assume that $\Phi$ is TD.
If $d=2$, then $\Phi$ is vacuously REC.
If $d \geq 3$ then $\Phi$ is REC by Proposition \ref{prop:vth}.
\end{proofof}

\begin{proofof}{\eqref{eq:iii}}
For $d=2$ use  Proposition \ref{prop:TDd2}.
For $d \geq 3$ use Propositions  \ref{prop:vth}, \ref{prop:QCHTD}.
\end{proofof}

\begin{proofof}{\eqref{eq:iv}}
Assume that  $\Phi$ is  QCHS.
Then for $0 \leq i,j\leq d$
such that $1<j-i<d$, we have $E_i A^* E_j =0$ and
$E^*_i A E^*_j=0$. By this and Propositions  \ref{prop:QCHTD}, \ref{prop:vthd2}   we see
that $E_0 A^* E_d=0$ if and only if $E^*_0 A E^*_d=0$.
The result follows.
\end{proofof}

\begin{proofof}{\eqref{eq:v}}
First assume that $\Phi$ is TD.
We show that $\Phi$ is QCHS.
For $d=2$  the conditions in Definition \ref{def:QCHS}(i),(ii) are vacuous,
so $\Phi$ is QCHS.
For $d \geq 3$
the Hessenberg system $\Phi$ is QCHS
by Propositions \ref{prop:vth} and  \ref{lem:JHL}(i),(ii).

Next assume that $\Phi$ is  QCHS. We show that $\Phi$ is TD.
By \eqref{eq:iv}, $\Phi$ is CHS or THS.
By this and \eqref{eq:i}, \eqref{eq:iii}
the Hessenberg system
$\Phi$ is TD.
\end{proofof}

\end{document}